\documentclass[sn-mathphys,Numbered]{sn-jnl}% Math and Physical Sciences Reference Style
\usepackage{graphicx}%
\usepackage{multirow}%
\usepackage{amsmath,amssymb,amsfonts}%
\usepackage{amsthm}%
\usepackage{mathrsfs}%
\usepackage[title]{appendix}%
\usepackage{xcolor}%
\usepackage{textcomp}%
\usepackage{manyfoot}%
\usepackage{booktabs}%
\usepackage{listings}%
\usepackage{mathtools}
\usepackage{tikz}
\usepackage{algorithm}
\usepackage{algorithmicx}
\usepackage{algpseudocode}
\usepackage{stmaryrd}

\let\vec\mathbf
\DeclarePairedDelimiter\norm{\lVert}{\rVert}
\DeclarePairedDelimiter\abs{\lvert}{\rvert}
\newcommand{\vkk}{\mathcal{V}^{K,k}}
\newcommand{\ekk}{\mathcal{E}^{K,k}}
\newcommand{\pkk}{\mathcal{P}^{K,k}}
\newcommand{\vk}{\mathcal{V}^{k}}
\newcommand{\ek}{\mathcal{E}^{k}}
\newcommand{\pk}{\mathcal{P}^{k}}
\theoremstyle{thmstyleone}%
\newtheorem{theorem}{Theorem}%  meant for continuous numbers
\newtheorem{proposition}[theorem]{Proposition}% 
\newtheorem{lemma}[theorem]{Lemma}% 

\theoremstyle{thmstyletwo}%
\newtheorem{remark}{Remark}%

\theoremstyle{thmstylethree}%

\begin{document}

\title[When rational functions meet virtual elements: The lightning Virtual Element Method]{When rational functions meet virtual elements:\\The lightning Virtual Element Method}


\author[1]{\fnm{Manuel Luigi} \sur{Trezzi}}\email{manuelluigi.trezzi01@universitadipavia.it}

\author*[2]{\fnm{Umberto} \sur{Zerbinati}}\email{umberto.zerbinati@maths.ox.ac.uk}

\affil[1]{\orgdiv{Dipartimento di Matematica}, \orgname{Uiniversità di Pavia}, \orgaddress{\street{Via Ferrata 5}, \city{Pavia}, \postcode{27100}, \country{Italy}}}

\affil[2]{\orgdiv{Mathematical Institute}, \orgname{University of Oxford}, \orgaddress{\street{Andrew Wiles Building}, \city{Oxford}, \postcode{OX2 6GG}, \country{United Kingdom}}}


\abstract{
We propose a lightning Virtual Element Method that eliminates the stabilisation term by actually computing the virtual component of the local VEM basis functions using a lightning approximation.
In particular, the lightning VEM approximates the virtual part of the basis functions using rational functions with poles clustered exponentially close to the corners of each element of the polygonal tessellation.
This results in two great advantages. First, the mathematical analysis of a priori error estimates is much easier and essentially identical to the one for any other non-conforming Galerkin discretisation.
Second, the fact that the lightning VEM truly computes the basis functions allows the user to access the point-wise value of the numerical solution without needing any reconstruction techniques.
The cost of the local construction of the VEM basis is the implementation price that one has to pay for the advantages of the lightning VEM method, but the embarrassingly parallelizable nature of this operation will ultimately result in a cost-efficient scheme almost comparable to standard VEM and FEM.
}




\maketitle

\section{Introduction}
Since its introduction in \cite{VEMBasic}, the Virtual Element Method (VEM) has been recognised as a valuable tool for the solution of partial differential equations (PDEs).
One of the key advantages of the VEM is that it allows any type of polygonal meshes, thanks to the introduction of a virtual component in the local basis functions.
We will discuss this aspect in a later section.
Curved meshes can also be used more naturally than iso-parametric finite elements \cite{CurvedVEM}.
Furthermore, given the rising interest in structure-preserving discretisation, it is worth mentioning that the VEM allows the mimicry of many interesting physical structures that arise at the level of the continuous PDE \cite{VEMStokes}.
For example, the VEM, unlike the standard finite element method (FEM), can be used to construct arbitrary low-order divergence-free $H^1$ conforming discretisations for Stokes flow \cite{VEMStokes}, without any restriction on the mesh.
The sophisticated mathematical infrastructure that allows the proof of accurate a priori error estimates, which are a key advantage of the FEM over other numerical schemes, can also be applied with some adjustments to the VEM.
Another feature of the VEM, perhaps the one of greatest importance, is the ability to easily produce discretisations with a high order of continuity, a property that is well known to be a weakness of the classical FEM.
It is also worth mentioning that VEM accuracy can be improved with the $p$ and $hp$ versions of VEM introduced in \cite{hpVEM}, clearly reflecting the intimate relation between the FEM and VEM.

The advantages just mentioned lead to the application of the VEM to a wide variety of problems from linear elasticity \cite{LavadinaEl,BMEl,LovadinaEl2}, fluid-dynamics \cite{Inc1,Inc2,Inc3}, fourth-order problems \cite{F1,F2,F3}, acoustic wave propagation and Helmholtz problem \cite{Wave1,Perugia}, and  magnetostatics \cite{M1,M2}.
While the mathematical infrastructure is what makes the VEM shine and reveals the advantages listed above, the challenges of the VEM are the practical aspects of its implementation.
The true novelty behind the VEM is the use of the so-called projection operators that allow assembling stiffness and mass matrices without the necessity of computing the virtual component of the basis functions.
Resorting to projection operators results in a lack of coercivity that requires the introduction of an additional stabilisation term in the weak formulation of the problem.

The lightning VEM proposed here eliminates the stabilisation term by actually computing, in an extremely efficient manner, the virtual component of the local VEM basis functions.
In particular, the lightning VEM approximates the virtual part of the basis functions using rational functions with poles clustered exponentially close to the corners of each element of the polygonal tessellation.
It is worth mentioning that in the literature other ideas have been proposed in order to get rid of the stabilisation term \cite{BerroneE2VEM,BerroneE2VEM2,EigLowOrder,AdaptiveE2VEM}.
The key difference between the lightning VEM and these methods is the fact that the lightning VEM does not make use of any projection operators.
Therefore, the mathematical analysis of a priori error estimates is much easier and essentially identical to that for any other non-conforming Galerkin discretisation.

Furthermore, the fact that the lightning VEM truly computes the basis function allows the user to access pointwise values of the numerical solution without the need for any reconstruction technique.
This is particularly appealing considering that a common reconstruction technique is based on polynomial interpolation and may require the triangulation of the polygonal mesh.

The local construction of the VEM basis is the implementation price that one has to pay for the advantages of the lightning VEM method, but as we will see the embarrassingly parallelizable nature of this operation will ultimately result in a cost-efficient scheme compared to standard VEM and FEM.

Before diving in the core aspects of the paper, the authors would like to stress the difference between the meaning of the word ``virtual'' in the context of the VEM and in the context of the lightning VEM.
In the first case``virtual'' refers to the fact that we have no point-wise knowledge of the basis functions in the interior of an element.
In the lightning VEM we compute a set of basis functions, for which we can access point-wise values, and ``virtual'' refers to the fact that we are approximating the standard virtual element space.

\section{Virtual Element Method} \label{sec:VEM}
For the sake of simplicity, we focus our attention on the Laplace problem.
The VEM has been applied to a large number of equations but we want to keep the focus on the simplest scenario. Given a polygonal open, bounded domain $\Omega\subset \mathbb{R}^2$ with boundary $\Gamma$, we consider the problem of finding a function $u:\Omega \to \mathbb R$ such that
\begin{equation} \label{eq:lap}
\left\{
\begin{aligned}
- \Delta u &= f \quad \text{in }\Omega \, , \\
u &= 0 \quad \text{on }\Gamma \, ,
\end{aligned}
\right .
\end{equation}
where $f \in L^2(\Omega)$ is the load term. It is well known that the standard weak formulation of \eqref{eq:lap} reads as
\begin{equation} \label{eq:lap-weak}
\left\{
\begin{aligned}
&\text{find $u \in V\coloneqq H^1_0(\Omega)$ such that: }\\
&a(u,v) = \langle f , v \rangle \quad \forall v \in H_0^1(\Omega) \, ,
\end{aligned}
\right .
\end{equation}
where the bilinear form $a(\cdot,\cdot): V \times V \to \mathbb R$ is defined as
\begin{equation}\label{eq:a-def}
\forall u,v \in V \, , \quad a(u,v) \coloneqq \int_\Omega \nabla u \cdot \nabla v \, .
\end{equation}
\subsection{The local space} We decompose the domain $\Omega$ in a tessellation $\mathcal{T}_h$ by a finite number of non-overlapping convex polygons $K$.
In particular, we assume that there exists a positive constant $\rho$ such that every element $K \in \mathcal{T}_h$ is star-shaped with respect to a ball of a radius greater or equal than $ \rho h_K$, where $h_K$ denotes the diameter of the element $K$.
Let $k \geq 1$ be the ``order" of the method. For every element $K \in \mathcal{T}_h$ we define the local virtual element space as
\begin{equation} \label{eq:discrete-space}
V_h^k(K) \coloneqq
\left \{
v_h \in H^1(K) \, \text{ such that } \,
\Delta v_h \in \mathbb P_{k-2}(K) \, \text{ and } \, 
v_h \in \mathbb B_K(\partial K)
\right \},
\end{equation}
where 
\begin{equation}
\mathbb B_k (\partial K) \coloneqq 
\left \{
 v_h \in \mathcal{C}^0(\partial K) \, \text{ such that } \,
 \forall e \in \partial K , \,
 v_h |_e \in \mathbb P_k(e) 
\right \},
\end{equation}
where $e$ denotes an edge of $\partial K$ and we have denoted $\mathbb P_{-1} \coloneqq \{0\}$. One can check that the following quantities represent a set of degrees of freedom for the space $V_h^k(K)$:
\begin{enumerate}
\item $\vkk$: the pointwise values of $v_h$ at the vertexes of the polygon $K$,
\item $\ekk$: the values of $v_h$ at $k-1$ internal points of a Gauss-Lobatto quadrature for every edge $e \in \partial K$,
\item $\pkk$: the moments $\dfrac{1}{| K |} \int_K v_h \, m_{\alpha\beta} \, {\rm d} K\, $, $\forall m_{\alpha\beta} \in \mathcal{M}_{k-2}(K)$ where $\mathcal{M}_{k-2}(K)$ is the set of monomials defined as
\begin{equation}
\label{eq:internal-dofs}
\mathcal{M}_{k-2}(K) \coloneqq \left\{
m_{\alpha\beta} \coloneqq \left( \dfrac{x - x_K}{h_K} \right)^\alpha
\left( \dfrac{y - y_K}{h_K} \right)^\beta
 \ \alpha,\beta\in\mathbb N\, , \alpha + \beta \leq k - 2
\right\}.
\end{equation}
\end{enumerate}
Details of the proof are in \cite{VEMBasic}. 
A graphic representation of the degrees of freedom is given in Figure \ref{f:dofs}. 
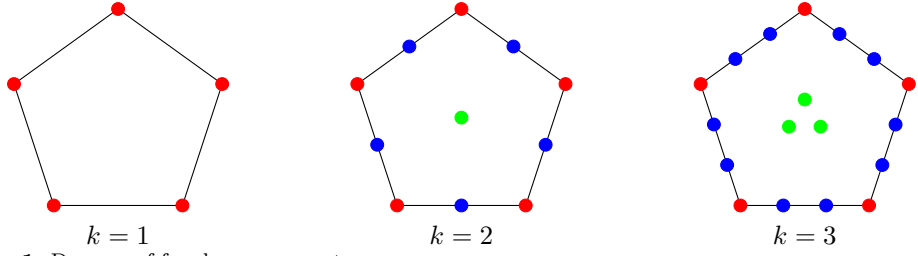
\begin{figure}\label{f:dofs}

\begin{minipage}{0.3\textwidth}
\centering
\begin{tikzpicture}[scale = 0.8]
\coordinate (A) at (90:1.8);
\coordinate (B) at (162:1.8);
\coordinate (C) at (234:1.8);
\coordinate (D) at (306:1.8);
\coordinate (E) at (378:1.8);
\draw (A) -- (B) -- (C) -- (D) -- (E) -- (A);
\draw[red, fill] (A) circle (3pt);
\draw[red, fill] (B) circle (3pt);
\draw[red, fill] (C) circle (3pt);
\draw[red, fill] (D) circle (3pt);
\draw[red, fill] (E) circle (3pt);
\end{tikzpicture} \\
$k=1$
\end{minipage}
\hfill
\begin{minipage}{0.3\textwidth}
\centering
\begin{tikzpicture}[scale = 0.8]
\coordinate (A) at (90:1.8);
\coordinate (B) at (162:1.8);
\coordinate (C) at (234:1.8);
\coordinate (D) at (306:1.8);
\coordinate (E) at (378:1.8);

\path (A) +(216:1.058) coordinate (A1);
\path (B) +(288:1.058) coordinate (B1);
\path (C) +(0:1.058) coordinate (C1);
\path (D) +(72:1.058) coordinate (D1);
\path (E) +(144:1.058) coordinate (E1);

\coordinate (01) at (0,0);

\draw (A) -- (B) -- (C) -- (D) -- (E) -- (A);

\draw[red, fill] (A) circle (3pt);
\draw[red, fill] (B) circle (3pt);
\draw[red, fill] (C) circle (3pt);
\draw[red, fill] (D) circle (3pt);
\draw[red, fill] (E) circle (3pt);

\draw[blue, fill] (A1) circle (3pt);
\draw[blue, fill] (B1) circle (3pt);
\draw[blue, fill] (C1) circle (3pt);
\draw[blue, fill] (D1) circle (3pt);
\draw[blue, fill] (E1) circle (3pt);

\draw[green, fill] (01) circle (3pt);
\end{tikzpicture} \\
$k=2$
\end{minipage}
\hfill
\begin{minipage}{0.3\textwidth}
\centering
\begin{tikzpicture}[scale=0.8]
\coordinate (A) at (90:1.8);
\coordinate (B) at (162:1.8);
\coordinate (C) at (234:1.8);
\coordinate (D) at (306:1.8);
\coordinate (E) at (378:1.8);

\path (A) +(216:0.705) coordinate (A1);
\path (B) +(288:0.705) coordinate (B1);
\path (C) +(0:0.705) coordinate (C1);
\path (D) +(72:0.705) coordinate (D1);
\path (E) +(144:0.705) coordinate (E1);

\path (A1) +(216:0.705) coordinate (A2);
\path (B1) +(288:0.705) coordinate (B2);
\path (C1) +(0:0.705) coordinate (C2);
\path (D1) +(72:0.705) coordinate (D2);
\path (E1) +(144:0.705) coordinate (E2);

\coordinate (O1) at (90:0.3);
\coordinate (O2) at (210:0.3);
\coordinate (O3) at (330:0.3);

\draw (A) -- (B) -- (C) -- (D) -- (E) -- (A);

\draw[red, fill] (A) circle (3pt);
\draw[red, fill] (B) circle (3pt);
\draw[red, fill] (C) circle (3pt);
\draw[red, fill] (D) circle (3pt);
\draw[red, fill] (E) circle (3pt);

\draw[blue, fill] (A1) circle (3pt);
\draw[blue, fill] (B1) circle (3pt);
\draw[blue, fill] (C1) circle (3pt);
\draw[blue, fill] (D1) circle (3pt);
\draw[blue, fill] (E1) circle (3pt);

\draw[blue, fill] (A2) circle (3pt);
\draw[blue, fill] (B2) circle (3pt);
\draw[blue, fill] (C2) circle (3pt);
\draw[blue, fill] (D2) circle (3pt);
\draw[blue, fill] (E2) circle (3pt);

\draw[green, fill] (O1) circle (3pt);
\draw[green, fill] (O2) circle (3pt);
\draw[green, fill] (O3) circle (3pt);
\end{tikzpicture} \\
$k=3$
\end{minipage}
\caption{Degrees of freedom on a pentagon.}
\end{figure}
Thanks to the definition of the degrees of freedom, it is possible to construct the projection operator $\Pi^{\nabla, K}_k : V_h^k(K) \to \mathbb P_k(K)$ defined as 
\begin{equation*}
\label{eq:Pn_k^K}
\left\{
\begin{aligned}
& \int_K \nabla \, p_k \cdot \nabla ( v_h - \, {\Pi}_{k}^{\nabla,K} v_h)\, {\rm d} K = 0 \quad  \text{for all $v_h \in V_h^k(K)$ and  $p_k \in \mathbb P_k(K)$,} \\
& \int_{\partial K}(v_h - {\Pi}_{k}^{\nabla, K}  v_h) \, {\rm d}s= 0 \, .
\end{aligned}
\right.
\end{equation*}
Indeed, thanks to integration by parts, we note that
\begin{equation}
\int_K \nabla \, p_k \cdot \nabla \, v_h
= 
- \int_K v_h \, \Delta p_k \, {\rm d} K
+
\int_{\partial K} \dfrac{p_k}{\partial n} \, v_h \, {\rm d} s.
\end{equation}
The first integral is known thanks to $\pkk$, for the second one we use $\ekk$ and $\vkk$.
The operator $\Pi^{\nabla,K}$ is an orthogonal projection into the space of polynomials with respect to the $H^1$ seminorm. This is the cornerstone around we can construct a discretisation of the bilinear form $a(\cdot, \cdot)$.
The global space is obtained by gluing together the local spaces
\begin{equation*}
\label{eq:vem global space}
V_h^k(\mathcal{T}_h) = \{v_h \in V \quad \text{s.t.} \quad v_h|_K \in V_h(K) \quad \text{for all $K \in \mathcal{T}_h$} \} \, ,
\end{equation*}
therefore obtaining a space characterized by the following degrees of freedom:
\begin{enumerate}
    \item $\vk$: the values of $v_h$ at the vertices;
    \item $\ek$: the values of $v_h$ at $k-1$ points on each edge $e$;
    \item $\pk$: the moments up to order $k-2$ for each element $E\in\mathcal{T}_h$.
\end{enumerate}

\subsection{The discrete problem} First, we decompose the global bilinear form into local contributions
\begin{equation}
a(u,v) = \sum_{K \in \mathcal{T}_h} a^K(u,v) \, .
\end{equation}
We point out that, except for a particular structures of the element $K$, we don't have an analytic expression for all the functions in $V_h^k(K)$. Hence, given two generic virtual functions, we are not able to compute the quantity
\begin{equation}
a^K(u_h,v_h) = \int_K \nabla u_h \cdot \nabla v_h \, {\rm d} K \, , \qquad u_h,v_h \in V_h^k(\mathcal{T}_h)
\end{equation}
We would like to construct a discrete bilinear form $a_h^K: V_h^k(K) \times V_h^k(K) \to \mathbb R$ that is computable for all the virtual functions and acts as a discrete counterpart of $a^K(\cdot, \cdot)$. The idea is to split the virtual functions as
\begin{equation}
v_h = \Pi^{\nabla, K}_k v_h + (I - \Pi^{\nabla, K}_k) v_h \, .
\end{equation}
Thanks to this choice, the bilinear form is split as
\begin{equation}
\begin{aligned}
a^K(u_h,v_h) &= a^K(\Pi^{\nabla, K}_k u_h,\Pi^{\nabla, K}_kv_h) + a^K((I - \Pi^{\nabla, K}_k) u_h, (I - \Pi^{\nabla, K}_k) v_h) \\
&+ a^K(\Pi^{\nabla, K}_k u_h, (I - \Pi^{\nabla, K}_k) v_h) + a^K((I - \Pi^{\nabla, K}_k) u_h,\Pi^{\nabla, K}_k v_h) \, .
\end{aligned}
\end{equation}
Thanks to orthogonality of $\Pi^{\nabla,K}_k$, the last two terms are equal to zero. The first term involves only polynomials hence is computable. This term is known in the VEM literature as the consistency term. The only thing that remains to be done is to handle the second term. The idea in the VEM framework is to replace it with a computable bilinear form $\mathcal{S}^K(\cdot,\cdot):V_h^k(K) \times V_h^k(K) \to \mathbb R$ that satisfies 
\begin{equation}
\label{eq:sEh}
\alpha_*|v_h|_{1,E}^2 \leq \mathcal S(v_h, v_h) \leq \alpha^* |v_h|_{1,E}^2 \, ,
\qquad \text{for all $v_h \in {\rm Ker}(\Pi^{\nabla,K}_k)$} \, ,
\end{equation}
for two positive uniform constants $\alpha_*$ and $\alpha^*$. This term is called the stability term. We define the bilinear form
\begin{equation}
a_h^K(u_h,v_h) \coloneqq a^K( \Pi^{\nabla,K}_k u_h, \Pi^{\nabla,K}_k v_h ) + \mathcal{S}^K( (I - \Pi^{\nabla,K}_k) u_h , (I - \Pi^{\nabla,K}_k) v_h) \, .
\end{equation}
The global bilinear form is obtained summing all the local contributions
\begin{equation}
a_h(u_h,v_h) \coloneqq \sum_{K \in \mathcal{T}_h} a^K_h(u_h, v_h) 
\end{equation}
It remains to discretise the load term. A standard choice is the following procedure:
\begin{enumerate}
    \item if $k \geq 2$, we replace $f$ with $f_h$ defined locally as the $L^2$-orthogonal projection into the space $\mathbb P_{k-2}(K)$. In detail, we consider
    \begin{equation}
    \langle f_h, v_h \rangle = \sum_{K \in \mathcal{T}_h} \int_K f_h \, v_h {\rm d}K = \sum_{K \in \mathcal{T}_h} \int_K f \, \Pi^{0,E}_{k-2} v_h {\rm d}K \, ;
    \end{equation}
    \item if $k=1$, we replace $f$ with a piecewise constant and define
    \begin{equation}
    \langle f_h, v_h \rangle = \sum_{K\in\mathcal{T}_h} \int_K P_0^K f \, \overline v_h {\rm d}K \, 
    \end{equation}
    where
    \begin{equation}
    \overline v_h = \dfrac{1}{n} \sum_{i=1}^n v_h(V_i) \, ,\quad V_i = \text{vertices of $K$}\, .
    \end{equation}
\end{enumerate}
The discrete problem reads as
\begin{equation} \label{eq:lapWeakDisct}
\left\{
\begin{aligned}
&\text{find $u_h \in V_h^k(\Omega_h)$ such that: }\\
&a_h(u_h,v_h) = \langle f_h , v_h \rangle \quad \forall v_h \in V_h^k(\Omega_h) \, .
\end{aligned}
\right .
\end{equation}

\begin{remark}
  If we desire an $H^1$ conforming discretisation the most common choice of virtual operator is the negative Laplacian as in \eqref{eq:discrete-space}.
The well-posedness of the local virtual element space is guaranteed even if we choose any other elliptic operator with sufficiently smooth coefficients.
Yet to recover the, optimal approximation property, it will be convenient also to assume that $\mathbb{P}_{k}$ is a subset of the co-domain of chosen the operator when applied to the space $\mathbb{P}_{k-2}$ with the constraint of having boundary data in $\mathbb{P}_k$.
A common misunderstanding when first approaching the VEM is to think in terms of Trefftz methods, i.e. to assume that the operator appearing in \eqref{eq:discrete-space} has to be the same as the one appearing in \eqref{eq:lap}.
We direct the reader interested in Trefftz methods to \cite{PaulTreffetz,Trefftz}.
We should think of the operator appearing in \eqref{eq:discrete-space} only in terms of the conformity we desire.

\end{remark}
\begin{remark}
  \label{rmk:polyLaplacian}
  In general, if we require an $H^k$ conforming discretisation the most common choice of operator in \eqref{eq:discrete-space} would be the polylaplacian $-\Delta^k$.
  Of course, we need to modify slightly our definition of local VEM space to accommodate the fact that we have to prescribe not only the value of the solution at the boundary but also the first $k-1$ normal derivatives.
  More detail on this topic can be found in \cite{F1,F2,F3,Polylapcian,PolyLaplacianPt2}.
\end{remark}

\begin{remark}
    An important observation is that the VEM only provides the value of the degrees of freedom as result. If one is interested in the value of the discrete solution outside of $\mathcal{V}^k$ and $\mathcal{E}^k$ a reconstruction technique has to be used.
    A common reconstruction technique is to triangulate each polygonal element and perform a linear interpolation using the nodal degrees of freedom. We will see later on that the lightning virtual element method avoids this issue.
\end{remark}

\section{The lightning approximation}
The key idea behind the lightning VEM is to consider a different local discrete space which is obtained by approximating $V_h^k(K)$ using rational functions with poles clustered exponentially close to the corners of $K$.
For the sake of clarity, we will focus our discussion on the lowest order case, $k=1$. We will discuss in a later section how to deal with the case $k>1$ and the case where higher conformity is required.
Thanks to the linearity of the Laplacian, if $K$ has $N_k$ edges, to construct $V_h^k(K)$ we are interested in solving $N_K$ problems of the form
\begin{align}
    \Delta \phi_i = 0 \;\;&\text{in}\; K,\label{eq:basisProblem}\\
    \phi_i= \varphi_i\;\; & \text{on}\;\partial K.\nonumber
\end{align}
where $\varphi_i$ are the functions in $\mathbb B_1(\partial K)$ that are equal to one on the $i$-th vertex and zero on the others.
If we are working with a structured rectangular mesh we can solve this problem exactly.
As the geometry of $K$ gets more complicated we need to find an efficient way of solving \eqref{eq:basisProblem}.
To address this problem we turn to the lightning Laplace solver presented in \cite{LighteningLaplace,LighteningLaplace2,LightningHelmholtz}.
The idea behind the lightning Laplace method is to search for an approximation $\hat{\phi}_i$ of $\phi_i$ 
 of the form
\begin{equation}
	\hat{\phi}_i = \Re\bigg\{\sum_{j=0}^{N_P}\frac{a_j}{z-z_j}+\sum_{j=0}^{N_Z}b_j(z-z_*)^j\bigg\},\label{eq:lightningFunctions} 
\end{equation} 
where $\{z_j\}_{j=1}^{N_P}$ and $z_*$ are points in the complex plane and $\Re$ denotes the real part of a complex number.
Finding the optimal coefficients, $\{z_j\}_{j=1}^{N_P}$ and $z_*$ to minimize $\norm{\hat{\phi}_i-\phi_i}_{L^\infty(K)}$ in general is a highly non-linear problem.
To transform this non-linear problem into a linear one the position of the points $\{z_j\}_{j=1}^{N_P}$ and $z_*$ are fixed only based upon the geometry of $K$.
In particular, the $N_P$ poles are clustered exponentially closer to the corners of the polygon $K$. Under this hypothesis, the following result was proven in \cite{LighteningStokes}:
\begin{lemma}
	Let $K$ be a convex polygon with corners $w_1,\dots,w_{N_K}$ and let $f$ be an analytic function on $K$ that is analytic on the interior of each side segment.
	Furthermore, we assume that $f$ can be analytically continued to a disk near any corner $w_k$ with a slit along the exterior bisector of the corner.
	Lastly, we assume that at each corner $w_k$ the function $f$ satisfies $f(z)-f(w_k)=\mathcal{O}(\norm{z-w_k}^\delta)$ as $z\to w_k$, for some $\delta>0$.
	Under these assumptions, there exists a rational function $r$ of the form
	\begin{equation}
		r = \Re\bigg\{\sum_{j=0}^{N_P}\frac{a_j}{z-z_j}+\sum_{j=0}^{N_Z}b_j(z-z_*)^j\bigg\},
	\end{equation}
	with $N_P$ poles $z_j$ only at points exponentially clustered along the exterior bisectors at the corners, such that the following approximation bound holds, for $C>0$:
	\begin{equation}
		\norm{f-r}_{L^\infty(K)} \leq C_K e^{-C\sqrt{N_P}}.
	\end{equation}
\end{lemma}
The previous lemma is a strong indication that the lightning Laplace scheme will be able to converge to an extremely accurate solution, yet it can not be applied in any straightforward manner to derive a priori error estimates with respect to the $H^1$ and $L^2$ norms.
An idea might be to mimic the reasoning presented in \cite{Yuji}, which can be used to derive an a priori error estimate on a least squares collocation method in terms of Lebesgue constant and the best approximation estimate presented above. The reason why this is not a viable path is that we are interested here in $H^1$ error estimates, which can not be produced using the type of $L^\infty$ bound presented in \cite{Yuji}.
To produce a priori error estimates for the overall scheme, we decide to proceed adaptively when it comes to the resolution of \eqref{eq:basisProblem} using the lightning Laplace method.
The exact algorithm used to compute the solution $\hat{\phi}_i$ to \eqref{eq:basisProblem} is presented in Algorithm \ref{alg:lightningLaplace} and we once again direct the reader interested in more detail to \cite{LighteningLaplace}.
\begin{algorithm}[h]
\caption{
	For computing the solution $\hat{\phi}_i^{(e)}$ of \eqref{eq:basisProblem}.}
	\label{alg:lightningLaplace}
	\begin{algorithmic}[1]
		\Require{$n_1,\;n_2,\;\dots$ such that $\sqrt{n_1},\;\sqrt{n_2},\;\dots$ is uniformly distributed, a tolerance $\varepsilon$.}
		\For{$n = n_1,n_2,\dots$}
			\State{Fix $N_P \approx nN_K$ and cluster the poles outside $K$ and exponentially close the corners;}
			\State{Fix $z_*$ in the interior of $K$ and fix $N_Z \approx N_K$ elements of the monomial basis;}
			\State{Choose $M=6N_P+6N_Z+1$ sample points on the boundary and clustered near the corners;}
			\State{Evaluate at the sample points $\varphi_i$ to obtain a matrix $A\in \mathbb{R}^{M\times (2N_P+2N_Z+1)}$ and $\vec{d}\in \mathbb{R}^{M}$;}
			\State{Solve the least squares problem $A\vec{x} \approx \vec{d}$ for the coefficients $\vec{x}=(\vec{a},\vec{b})$;}
			\State{Construct the function $\hat{\phi}_i = \Re\bigg\{\sum_{j=0}^{N_P}\frac{a_j}{z-z_j}+\sum_{j=0}^{N_Z}b_j(z-z_*)^j\bigg\}$ using the coefficients $\vec{a}$ and $\vec{b}$;}
			\If{$\norm{\hat{\phi}_i-\varphi_i}_{H^{1/2}({\partial K})}\leq \varepsilon$}
				\State{\textbf{break;}}
			\EndIf
		\EndFor
		\Return{$\hat{\phi}_i$.}
	\end{algorithmic}
\end{algorithm}
We then observe that since  $\sum_{j=0}^{N_P}\frac{a_j}{z-z_j}+\sum_{j=0}^{N_Z}b_j(z-z_*)^j$ is analytic in the convex hull described by the vertex of any polygonal element $K$ of the tessellation $\mathcal{T}_h$ where $z_j$ are the vertices of $K$, then its real part is a harmonic function.
We direct the reader unfamiliar with this notion to \cite{Gilardi,Ahlfors}.
Since $\hat{\phi}_i$ is harmonic we know that the function defined as $\hat{\phi}_i-\phi_i$ is harmonic and by applying  boundary regularity estimates for harmonic functions we obtain the following result.
\begin{proposition}
	Let $\phi_i$ be as in \eqref{eq:basisProblem} and $\hat{\phi}_i$ be the outcome of Algorithm \ref{alg:lightningLaplace}. Then 
	\begin{equation}
		\norm{\hat{\phi}_i-\phi_i}_{H^1(K)}\leq C_1\varepsilon, \qquad \norm{\hat{\phi}_i -\phi_i}_{L^{\infty}(\partial K)}\leq C_2\varepsilon,\label{eq:lightningErrorEstimates}
	\end{equation}
	where the constants $C_1$ and $C_2$ only depend on $\Omega$ and $\partial\Omega$.
\end{proposition}
\begin{proof}
	Since we assumed the elements $K$ of our tessellation $\mathcal{T}_h$ are polygonal, proof for the boundary estimates used to produce the first inequality in \eqref{eq:lightningErrorEstimates} can be found in \cite{Grisvard,GilbargTrudinger}.
	The second estimate comes from the Morrey's inequality on a parametrization of the boundary \cite{Evans}.
\end{proof}
From now on we will adopt the notation $\hat{V}_{h,\varepsilon}(K)$ to express the discrete function space that is obtained from the solution of \eqref{eq:basisProblem} using the lightning Lapalce method on the element $K$ of the tessellation and $\hat{V}_{h,\varepsilon}(\mathcal{T}_h)$ to denote the global space constructed starting from the various $\hat{V}_{h,\varepsilon}(K)$.
\begin{remark}
    It is clear that a fundamental step in the lightning VEM is solving the least squares problem
    \begin{equation}
        A\vec{x}\approx \vec{d}\label{eq:leastSquare}
    \end{equation}that appears in Algorithm \ref{alg:lightningLaplace}. In particular the least squares problem \eqref{eq:leastSquare} is observed to be terribly ill-conditioned. Yet is still possible to solve \eqref{eq:leastSquare} with great accuracy if standard regularisation are adopted \cite{LighteningLaplace}. In particular in \cite{Herremans2023} it is discussed the effect that the pole at infinity in \eqref{eq:lightningFunctions} have on the $2$-norm of $\vec{x}$. In fact bounding the $2$-norm of $\vec{x}$ will guarantee the solvability of \eqref{eq:leastSquare} to a high degree of accuracy in spite its ill-conditioned nature. An other route to solve \eqref{eq:leastSquare} to a great deal of accuracy, inspite its ill-conditioned nature, is to use the Vandermonde with Arnoldi algorithm, introduced in \cite{VA} as discussed in \cite{Yuji,LighteningStokes}.
\end{remark}
\section{The lightning virtual element method}

We begin by observing that the degrees of freedom $\mathcal{V}_h$ and $\mathcal{E}_h$ are not decoupled in adjacent elements.
In fact if $\hat{\phi}_i$ had been a polynomial on the edges of the element $K$ as in the standard VEM space $V_h(K)$, then this would have ensured the continuity of the function $\hat{\phi}_i$ across the elements interface.
Instead, since each $\hat{\phi}_i$ is a rational function on the edges of the element $K$ even if over adjacent elements $\hat{\phi}_i$ would have the same value on the edges and vertex degrees of freedom we still might have a jump across the elements interface, as depicted Figure \ref{fig:continuity}.
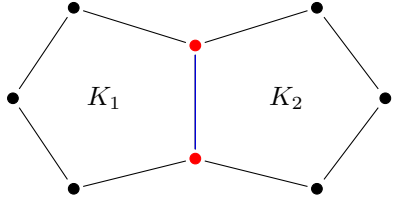
\begin{figure}
    \centering
    \caption{We know that the basis function $\hat{\phi}_{i,K_1}$ and $\hat{\phi}_{i,K_2}$ corresponding to the $i$-th vertex and constructed respectively on $K_1$ and $K_2$, match at the degrees of freedom here denoted in red. Yet we have no guarantee that $\hat{\phi}_1$ and $\hat{\phi}_2$ are continuous along the blue edge.}
    \vspace{1cm}
    \label{fig:continuity}
    \begin{tikzpicture}[scale = 0.4,rotate=90]
        \node  (A) at (0, 6) {};
        \node  (B) at (-3, 4) {};
        \node  (C) at (-2, 0) {};
	\node  (D) at (1.75, 0) {};
	\node  (E) at (3, 4) {};
	\node  (F) at (-3, -4) {};
	\node  (G) at (3, -4) {};
	\node  (H) at (0, -6.25) {};
    \node  (C1) at (0, 3) {$K_1$};
    \node (C2) at (0,-3) {$K_2$};
        
        \draw (A) -- (B) -- (C) -- (D) -- (E) -- (A);
        \draw (C) -- (F) -- (H) -- (G) -- (D) -- (C);
        \draw [blue] (C) -- (D) ;
        
        \draw[black, fill] (A) circle (5pt);
        \draw[black, fill] (B) circle (5pt);
        \draw[red, fill] (C) circle (5pt);
        \draw[red, fill] (D) circle (5pt);
        \draw[black, fill] (E) circle (5pt);
        \draw[black, fill] (F) circle (5pt);
        \draw[black, fill] (G) circle (5pt);
        \draw[black, fill] (H) circle (5pt);
    \end{tikzpicture} 
\end{figure}
It is important to notice that this jump is bounded by adaptively solving the lightning approximation problem in Algorithm \ref{alg:lightningLaplace}, i.e. $\norm{\hat{\phi}_i-\phi_i}_{H^{1/2}({\partial K})}\leq \varepsilon$.
Since the function in $\hat{V}_{h,\varepsilon}(\mathcal{T}_h)$ are no longer continuous we need to consider a broken version of the bilinear form $a$ in \eqref{eq:a-def}, i.e.
\begin{equation}
	\hat{a}(\hat{u}_h,\hat{v}_h) = \sum_{K\in \mathcal{T}_h} a^K(\hat{u}_h,\hat{v}_h)
\end{equation}
This bilinear form on $H^1(\Omega)$ is still symmetric positive definite with respect to the broken $H^1(\Omega)$ norm and for the continuous solution to the problem we have that $\hat{a}(u,v)=a(u,v)$ since $u,v\in H^1(\Omega)$.
Furthermore, in $\hat{V}_{h,\varepsilon}(\mathcal{T}_h)$ the kernel of the bilinear form $\hat{a}(\cdot,\cdot)$ are only functions that are constant on each element of the tessellation, using the fact that such constants must have the same value on the vertices of neighboring elements we have that the boundary conditions impose that $\hat{a}(\cdot,\cdot)$ has a trivial kernel.
We now derive an a priori error estimate using the following Lemma from \cite{BrennerFEM}.
\begin{lemma}
    Let $\hat{a}(\cdot,\cdot)$ be a symmetric positive definite bilinear form on $H^1(\Omega)+\hat{V}_{h,\varepsilon}(\mathcal{T}_h)$ which reduces to $a(\cdot,\cdot)$ on $H^1(\Omega)$.
    We will further assume that $u$ is the solution of \eqref{eq:lap-weak} and $\hat{u}_h$ is such that:
    \begin{equation}
        \hat {a} (\hat{u}_h,\hat{v}_h) = \langle f,\hat{v}_h\rangle \textrm{ for all } \hat{v}_h \in \hat{V}_{h,\varepsilon}(\mathcal{T}_h),
    \end{equation}
    where in both the previous variation equation and \eqref{eq:lap-weak} we assume $f\in L^2(\Omega)\cap \hat{V}_{h,\varepsilon}(\mathcal{T}_h)^*$.
    Under this hypothesis the following best approximation estimate holds,
    \begin{equation}
        \norm{u-\hat{u}_h}_h \leq \underset{\hat{v}_h \in \hat{V}_{h,\varepsilon}}{\inf}\norm{u-\hat{v}_h}_h + \underset{\hat{w}_h \in \hat{V}_{h,\varepsilon}}{\sup}\frac{\abs{\hat{a}(u-\hat{u}_h,\hat{w}_h)}}{\norm{\hat{w_h}}_h},\label{eq:StrangLemma}
    \end{equation}
    where $\norm{\cdot}_h := \sqrt{\hat a (\cdot,\cdot)}$.
\end{lemma}
To apply the previous lemma, in order to obtain an a priori error estimate, we need first to provide an estimate for the second term in the right-hand side of \eqref{eq:StrangLemma}.
\begin{align}
    \hat a (u-\hat{u}_h,\hat{w}_h )&=
    \sum_{K\in \mathcal{T}_h} (\nabla u, \nabla\hat{w}_h)_{L^2(\Omega)}-(f, \hat{w}_h)_{L^2(\Omega)}\\
    &= \sum_{K\in \mathcal{T}_h} \int_{\partial K} (\nabla u \cdot \vec{n})  \hat{w}_h = \sum_{e \in \mathscr{E}_h} \int_e \nabla u \cdot \llbracket \hat{w}_h \rrbracket \, ,
\end{align}
where $\mathscr{E}_h$ are the internal edges of the tessellation $\mathcal{T}_h$ and $\llbracket \hat{w}_h \rrbracket$ denotes the vector jump of $\hat w_h$ across the edge $e$, i.e. 
\begin{equation}
\llbracket \hat{w}_h \rrbracket = \hat w_h\big\lvert_{K_1} \vec{n}_1 + \hat w_h\big\lvert_{K_2} \vec{n}_2 \, .
\end{equation}
We now rewrite the last term using the basis functions $\hat{\phi}_i \big \lvert_e$, 
\begin{equation}
    \sum_{e \in \mathscr{E}_h} \int_e \nabla u \cdot \llbracket \hat{w}_h \rrbracket
    \leq
    \sum_{e \in \mathscr{E}_h} \sum_{i=1}^{N_K} \int_e \hat{w}_i \big \lvert_e \nabla u \cdot \llbracket \hat{\phi}_i\rrbracket\leq \sum_{K\in \mathcal{T}_h} \varepsilon \norm{\hat{w}_h}_{h,K}\abs{u}_{H^2(K)},
\end{equation}
where the last inequality comes from the trace theorem combined with Cauchy-Schwarz inequality and the fact that we have constructed the basis function $\hat{\phi}^{(e)}_i$ in such a way that $\llbracket \hat{\phi}^{(e)}_i \rrbracket \leq  \hat{C}\varepsilon$, where $\hat{C}$ only depends on the shape of all the elements in the tessellation but not on their size.
Therefore assuming that each element in the tessellation is convex, so  that we can use the usual elliptic regularity result, we rewrite \eqref{eq:StrangLemma} as
\begin{equation}
    \norm{u-\hat{u}_h}_h \leq \!\!\underset{\hat{v}_h \in \hat{V}_{h,\varepsilon}}{\inf}\norm{u-\hat{v}_h}_h + \sum_{K\in \mathcal{T}_h} \hat{C}\varepsilon \norm{f}_{L^2(K)}\leq  \!\!\underset{\hat{v}_h \in \hat{V}_{h,\varepsilon}}{\inf}\norm{u-\hat{v}_h}_h + \hat{C}\varepsilon \norm{f}_{L^2(\Omega)}.
\end{equation}
We are therefore left estimating the best approximation property for objects that are in $\hat{V}_{h,\varepsilon}(\mathcal{T}_h)$.
If we denote $\Pi_k u$ the standard VEM interpolant, \cite{VEMBasic}, then we can write $\Pi_k u$ in terms of the basis functions $\phi_i$, i.e. $\Pi_k u = \sum_{K\in\mathcal{T}_h}\sum_{i=1}^{N_K} u_i \phi_i$ and introduce $\hat{\Pi}_k u = \sum_{K\in\mathcal{T}_h}\sum_{i=1}^{N_K} u_i \hat{\phi}_i$.
\begin{align}
    &\norm {\Pi_k u -\hat{\Pi}_j u }_h \leq \sum_{K\in \mathcal{T}_h}\sum_{i=1}^{N_K} \abs{u_i}\norm {\phi_i -\hat{\phi}_i}_{H^1(K)} \leq \sum_{K\in \mathcal{T}_h}\sum_{i=1}^{N_K} \abs{u_i} \hat{C}\varepsilon \leq \norm{f}_{L^2(\Omega)}\hat{C}\varepsilon
\end{align}
Therefore using the previous estimate together with the approximation estimates of the standard VEM interpolant we obtain
\begin{align}
    \norm{u-\hat{\Pi}_ku}_h&\leq \norm{u-{\Pi}_ku}_h+\norm{\Pi_k u-\hat{\Pi}_ku}_h \\  &\leq C(\Omega)h^{\max\{k,m-1\}} \abs{u}_{H^m(\Omega)}+\norm{f}_{L^2(\Omega)}\hat{C}\varepsilon.
\end{align}
Substituting this last estimate in the infimum appearing in the right hand side of \eqref{eq:StrangLemma}, we get
\begin{equation}
    \norm{u-\hat{u}_h}_h \leq C(\Omega)h^{\max\{k,m-1\}} \abs{u}_{H^m(\Omega)}+\norm{f}_{L^2(\Omega)}\hat{C}\varepsilon.
\end{equation}
\section{Numerical experiments} To investigate the behavior of the method, we consider a family of model problems in the unit square $\Omega = [0,1] \times [0,1]$ with as analytic solution the function
\begin{equation}
u(x,y) \coloneqq \sin(\pi x) \sin( \pi y) + \log(1 + xy) \, .
\end{equation} 
The meshes $\mathcal{T}_h$ that we will consider are centroidal Voronoi tessellations of the unit square. An example of a Voronoi tesselation is given in Figure \ref{fig:Voroni}.

\begin{figure}[H]\label{fig:Voroni}
\centering
\includegraphics[scale=0.5]{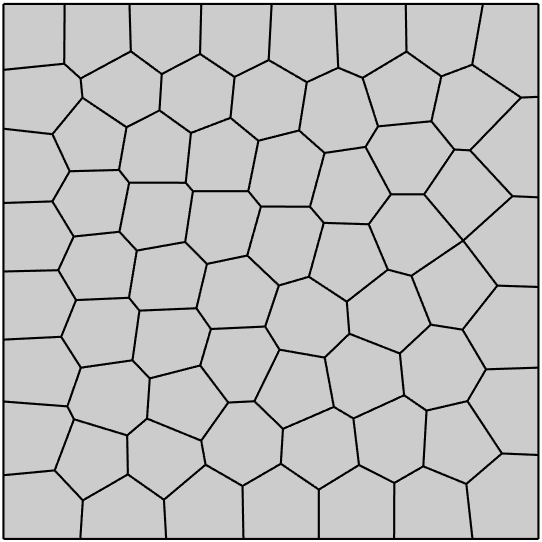}
\caption{A Voronoi tesselation made by 64 polygons.}
\label{fig:enter-label}
\end{figure}

Before presenting the cases of interest and the numerical results, we point out that in the usual VEM framework, the errors in the $L^2-$norm and $H^1-$seminorm are replaced by the following quantities:
\begin{equation}
e_{H^1} := \sqrt{\sum_{E\in\mathcal{T}_h}\left\|\nabla(u-\Pi_k^\nabla u_h)\right\|^2_{0,E}}\, \qquad
e_{L^2} := \sqrt{\sum_{E\in\mathcal{T}_h}\left\|(u-\Pi_k^\nabla u_h)\right\|^2_{0,E}}\,.
\end{equation}

This choice is due to not knowing the analytic expression of the virtual functions. Using lighting approximation, we can access the pointwise values of the approximated virtual element functions. This permits to compute the local errors using a quadrature formula without projecting into the space of polynomials. We therefore use the usual definitions of the $H^1$ and  $L^2$ errors.
The source code to reproduce all the numerical experiment presented in this section can be found in \cite{GitHub}.

\paragraph{The Laplace problem} As first model problem, we consider the PDE
\begin{equation}
\begin{aligned}
    - \Delta u  &= f \quad \text{in }\Omega \, ,\\
    u &\equiv g \quad \text{on } \partial \Omega \, .
  \end{aligned}
\end{equation}
We start from this equation since it represents the simplest elliptic PDE that one can consider. Considering a sequence of Voronoi tesselations that quadruples the number of polygons, we obtain the orders of convergence represented in Figure \ref{fig:caso1}, with $k=1$.

\begin{figure}
\centering
\includegraphics[scale=0.5]{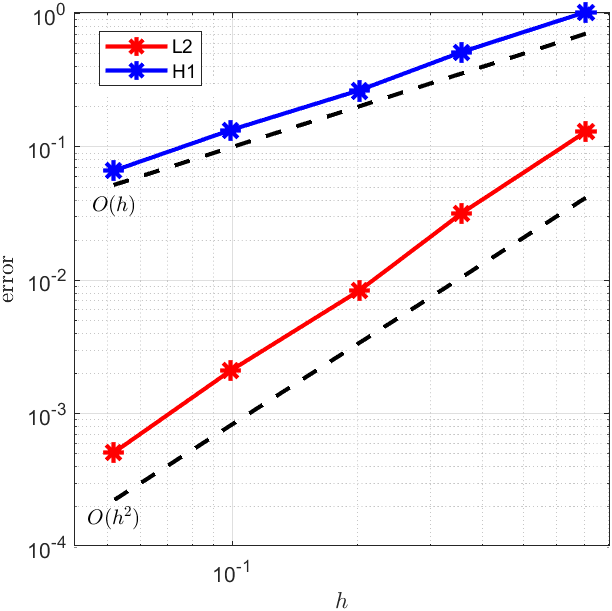}
\caption{Convergence results for the Laplace problem.}
\label{fig:caso1}
\end{figure}
We observe that we have achieved the expected order of convergence for this equation. 
In particular, the error in $L^2$-norm and $H^1-$seminorm decay as $\mathcal{O}(h^2)$ and $\mathcal{O}(h)$, respectively.

\paragraph{The diffusion-reaction problem} We add to the previous equation a reaction term. We obtain the following PDE
\begin{equation}
\begin{aligned}
    - \epsilon \Delta u  + \gamma u &= f \quad \text{in }\Omega \, ,\\
    u &\equiv g \quad \text{on } \partial \Omega \, ,
  \end{aligned}
\end{equation}
where $\epsilon>0$ and $\gamma \in L^\infty(\Omega)$ is a bounded non-negative function. This problem is of interest because the usual VEM approach to discretising the $L^2-$scalar product is to construct the $L^2-$orthogonal projection operator $\Pi^{0,E}_k: V_h^k(E) \to \mathbb P_k(E)$. This is not possibile with the standard definition of the virtual element space. To overcome this difficulty, the idea is to use the enhanced virtual element space defined as
\begin{equation}
\label{eq:enhanced}
\begin{aligned}
\tilde V_h(E) = 
\bigl\{
v_h \in H^1(E) \cap  C^0(\partial E) \quad \text{:} \quad 
 v_h|_e  \in \mathbb P_k (e) \quad &\text{for all $e \in \partial E$,} 
\bigr .
\\
\bigl .
 \Delta v_h \in \mathbb P_k(E) \,, \,
\langle v - \Pi^{\nabla,E}_k v, \, \widehat{p}_k \rangle = 0 \quad 
&\text{for all $\widehat{p}_k \in \mathbb P_k(E) / \mathbb P_{k-2}(E)$}
\bigr \} \,.
\end{aligned}
\end{equation}

Thanks to the lighting approximation, we do not need to project the virtual functions into the space of polynomials. This implies that we do not have to change the definition of the discrete space. This gives benefits also from the theoretical point of view.
For the numerical tests, we set $\epsilon = \gamma = 1$. The results are shown in Figure \ref{fig:caso2} and we observe that also for this equation we recovered the expected order of convergence.

\begin{figure}
\centering
\includegraphics[scale=0.5]{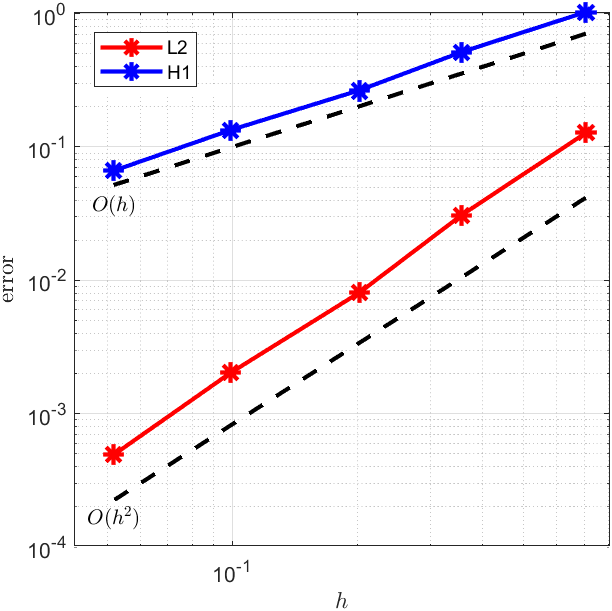}
\caption{Convergence results for the diffusion-reaction problem.}
\label{fig:caso2}
\end{figure}

\paragraph{The advection-diffusion-reaction problem} As the last model problem, we consider the advection-diffusion-reaction problem given by
\begin{equation}
\begin{aligned}
    - \epsilon \Delta u  + \beta \cdot \nabla u + \gamma u &= f \quad \text{in }\Omega \, ,\\
    u &\equiv g \quad \text{on } \partial \Omega \, .
  \end{aligned}
\end{equation}
where $\epsilon$ and $\gamma$ are as in the previous case and $\beta \in [W^{1,\infty}(\Omega)]^2$ with $\text{div} \beta = 0$. We point out that the bilinear form associated to the advective field
\begin{equation}
b(u,v) \coloneqq \int_\Omega (\beta \cdot \nabla u ) \, v {\rm d} \Omega \, ,
\end{equation}
is skew-symmetric. When we discretise $b(\cdot, \cdot)$ by inserting the projections onto the polynomial space, we lose this property.
Instead, to preserve the skew-symmetry property, we discrtize
\begin{equation}
b^{\text{skew}}(u,v) \coloneqq \dfrac{1}{2} \bigl( b(u,v) - b(v,u) \bigr) \, . \label{eq:skew}
\end{equation}
Using the lighting approximation, we avoid this difficulty and do not require \eqref{eq:skew}.
We select the same solution of the previous case and we set
$$
\beta(x,\,y) := \left[\begin{array}{r}
-2\,\pi\,\sin(\pi\,(x+2\,y))\\
\pi\,\sin(\pi\,(x+2\,y))
\end{array}\right]\, ,
$$
with $\epsilon = \sigma = 1$. To overcome problems related to the hyperbolicity of the advection term, we have assumed that we are not in an advection-dominated regime.
The numerical results are represented in Figure \ref{fig:caso3}. In Table \ref{tab:Time} we compare the average local assembly time between a vanilla VEM implementation and the lightning VEM method, clearly as the  mesh becomes finer the lightning VEM method is outperformed by the standard VEM method because we need to solve \eqref{eq:basisProblem} to a greater accuracy.
\begin{figure}
\centering
\includegraphics[scale=0.25]{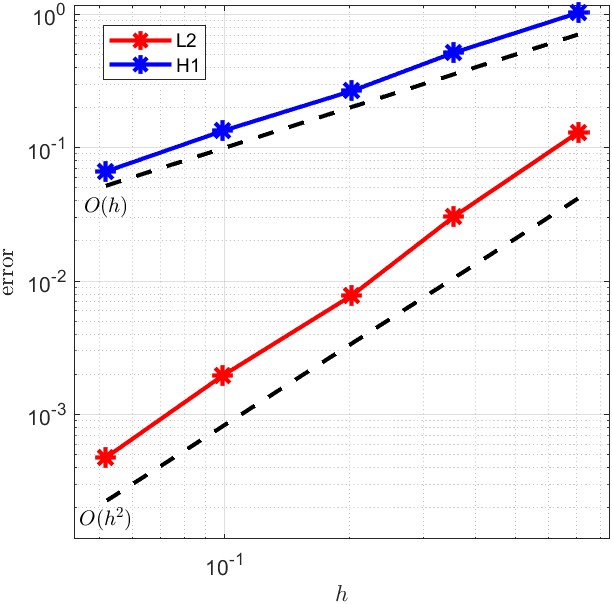}
\caption{Convergence results for the advection-diffusion-reaction problem}
\label{fig:caso3}
\end{figure}
\begin{table}[h]
    \caption{A comparison between a vanilla VEM implementation and the lightning VEM implementation, of the average time (in seconds) taken by the assembly of the local matrix for different number of elements.}\label{tab:Time}
    \begin{tabular*}{\textwidth}{@{\extracolsep\fill}l|ccccc}
    \toprule%
    N & 4 & 16 & 64 & 256 & 1024 \\
    \midrule
    Vanilla VEM&4.613250e-03 & 2.031250e-03 & 2.201594e-03 & 1.102801e-03 & 1.036618e-03 \\
    Lightning VEM & 3.679250e-03 & 3.221188e-03 & 6.074953e-03 & 9.153375e-03 & 1.845604e-02 \\
    \botrule
    \end{tabular*}
\end{table}

\section{Extensions of the lightning virtual element method}
As seen in the literature about the VEM it is possible to generalise a vast variety of ideas developed on the FEM framework also the lightning VEM.
In this section we would like to detail some of these extensions that we plan to investigate in the near future.
The authors would like to point out that most of these extensions are only possible thanks to the large body of work on the lightening and AAA approximation by Nick Trefethen, Yuji Nakatsukasa and colleagues.
\begin{enumerate}
	\item \textbf{High order discretisation}, we have focused our attention on the lowest order VEM, i.e. the one that has only degree of freedom the nodal evaluation in the vertices of our mesh.
	We would like to point out that extending the lightning VEM idea to higher order approximation is an easy task. In particular, it is enough to consider as \eqref{eq:discrete-space} a modified version similar to the lowest degree case, i.e.
	\begin{equation}
		V_h^k(K) = \Big\{u \in C^2(K)\;:\; u_{|e} \in \mathbb{P}_{k}(e)\; \text{and} \; \Delta (u-\pi) = 0\Big\},
	\end{equation}
	where $\pi\in \mathbb{P}_{k}$ is determined, in a non-unique manner by the internal moment degrees of freedom.
	\item \textbf{$C^k$ discretisation}, as mentioned in Reamark \ref{rmk:polyLaplacian} it is possible to obtain discretisation with a high degree of conformity susbtituting \eqref{eq:discrete-space} with
	\begin{align}
		V_h^k(K) = \Big\{u \in C^{k+2}(K)\;:\; 
        \partial_n^j u_{|e} \in \mathbb{P}_{2k-j}(e), \;\;
        &u_{|e} \in \mathbb{P}_{2k-1}(e),\;\; \Delta^k u = 0,\;\; \\&j=2,\dots, p\Big\}.
	\end{align}
	In this case we can still proceed constructing a set of basis functions in the same spirit as the lightening VEM method but following \cite{LighteningStokes}.
	\item \textbf{Curved elements}, a careful reader might have notice that an essential requirement for the lightning virtual element method is that the tessellation $\mathcal{T}_h$ is made by polygonal elements.
	In fact in order to transform the non-linear problem of fitting \eqref{eq:lightningFunctions} in a linear one we made the a priori choice to cluster the poles of \eqref{eq:lightningFunctions} exponentially close to the corner of the polygonal elements.
	How do we choose the position of the poles if we have a curved element? This question has been addressed in \cite{AAA} and we plan to use a similar reasoning to extend the lightning virtual element method to tessellations with curved elements.
	\item \textbf{Eigenvalue problems}, it has been observed in \cite{EigReview,EigParameterDep} that the stabilisation term plays a harmful role in the discretisation of eigenvalue problems using a virtual element discretisation. Possible fixes have been proposed in \cite{EigLowOrder,Gardini}.
	We plan to study the role that the absence of a stabilisation term plays when discretising an eigenvalue problem using the lightning virtual element method.
\end{enumerate}

\section*{Acknowledgments}
We would like to acknowledge Astrid Herremans, Nick Trefethen, Yuji Nakatsukasa and Patrick E. Farrell for the invaluable conversation and continuous help while writing this paper.

\bibliography{sn-bibliography}% common bib file


\begin{thebibliography}{43}
\ifx \bisbn   \undefined \def \bisbn  #1{ISBN #1}\fi
\ifx \binits  \undefined \def \binits#1{#1}\fi
\ifx \bauthor  \undefined \def \bauthor#1{#1}\fi
\ifx \batitle  \undefined \def \batitle#1{#1}\fi
\ifx \bjtitle  \undefined \def \bjtitle#1{#1}\fi
\ifx \bvolume  \undefined \def \bvolume#1{\textbf{#1}}\fi
\ifx \byear  \undefined \def \byear#1{#1}\fi
\ifx \bissue  \undefined \def \bissue#1{#1}\fi
\ifx \bfpage  \undefined \def \bfpage#1{#1}\fi
\ifx \blpage  \undefined \def \blpage #1{#1}\fi
\ifx \burl  \undefined \def \burl#1{\textsf{#1}}\fi
\ifx \doiurl  \undefined \def \doiurl#1{\url{https://doi.org/#1}}\fi
\ifx \betal  \undefined \def \betal{\textit{et al.}}\fi
\ifx \binstitute  \undefined \def \binstitute#1{#1}\fi
\ifx \binstitutionaled  \undefined \def \binstitutionaled#1{#1}\fi
\ifx \bctitle  \undefined \def \bctitle#1{#1}\fi
\ifx \beditor  \undefined \def \beditor#1{#1}\fi
\ifx \bpublisher  \undefined \def \bpublisher#1{#1}\fi
\ifx \bbtitle  \undefined \def \bbtitle#1{#1}\fi
\ifx \bedition  \undefined \def \bedition#1{#1}\fi
\ifx \bseriesno  \undefined \def \bseriesno#1{#1}\fi
\ifx \blocation  \undefined \def \blocation#1{#1}\fi
\ifx \bsertitle  \undefined \def \bsertitle#1{#1}\fi
\ifx \bsnm \undefined \def \bsnm#1{#1}\fi
\ifx \bsuffix \undefined \def \bsuffix#1{#1}\fi
\ifx \bparticle \undefined \def \bparticle#1{#1}\fi
\ifx \barticle \undefined \def \barticle#1{#1}\fi
\bibcommenthead
\ifx \bconfdate \undefined \def \bconfdate #1{#1}\fi
\ifx \botherref \undefined \def \botherref #1{#1}\fi
\ifx \url \undefined \def \url#1{\textsf{#1}}\fi
\ifx \bchapter \undefined \def \bchapter#1{#1}\fi
\ifx \bbook \undefined \def \bbook#1{#1}\fi
\ifx \bcomment \undefined \def \bcomment#1{#1}\fi
\ifx \oauthor \undefined \def \oauthor#1{#1}\fi
\ifx \citeauthoryear \undefined \def \citeauthoryear#1{#1}\fi
\ifx \endbibitem  \undefined \def \endbibitem {}\fi
\ifx \bconflocation  \undefined \def \bconflocation#1{#1}\fi
\ifx \arxivurl  \undefined \def \arxivurl#1{\textsf{#1}}\fi
\csname PreBibitemsHook\endcsname

%%% 1
\bibitem[\protect\citeauthoryear{Beir\~{a}o~da Veiga et~al.}{2013}]{VEMBasic}
\begin{barticle}
\bauthor{\bsnm{Veiga}, \binits{L.}},
\bauthor{\bsnm{Brezzi}, \binits{F.}},
\bauthor{\bsnm{Cangiani}, \binits{A.}},
\bauthor{\bsnm{Manzini}, \binits{G.}},
\bauthor{\bsnm{Marini}, \binits{L.D.}},
\bauthor{\bsnm{Russo}, \binits{A.}}:
\batitle{{Basic principles of virtual element methods}}.
\bjtitle{Math. Models Methods Appl. Sci.}
\bvolume{23}(\bissue{1}),
\bfpage{199}--\blpage{214}
(\byear{2013})
\doiurl{10.1142/S0218202512500492}
\end{barticle}
\endbibitem

%%% 2
\bibitem[\protect\citeauthoryear{Beir\~{a}o~da Veiga et~al.}{2019}]{CurvedVEM}
\begin{barticle}
\bauthor{\bsnm{Veiga}, \binits{L.}},
\bauthor{\bsnm{Russo}, \binits{A.}},
\bauthor{\bsnm{Vacca}, \binits{G.}}:
\batitle{{The virtual element method with curved edges}}.
\bjtitle{ESAIM Math. Model. Numer. Anal.}
\bvolume{53}(\bissue{2}),
\bfpage{375}--\blpage{404}
(\byear{2019})
\doiurl{10.1051/m2an/2018052}
\end{barticle}
\endbibitem

%%% 3
\bibitem[\protect\citeauthoryear{{Beirão da Veiga} et~al.}{2022}]{VEMStokes}
\begin{barticle}
\bauthor{\bsnm{{Beirão da Veiga}}, \binits{L.}},
\bauthor{\bsnm{Dassi}, \binits{F.}},
\bauthor{\bsnm{{Di Pietro}}, \binits{D.A.}},
\bauthor{\bsnm{Droniou}, \binits{J.}}:
\batitle{{Arbitrary-order pressure-robust DDR and VEM methods for the Stokes
  problem on polyhedral meshes}}.
\bjtitle{Computer Methods in Applied Mechanics and Engineering}
\bvolume{397},
\bfpage{115061}
(\byear{2022})
\doiurl{10.1016/j.cma.2022.115061}
\end{barticle}
\endbibitem

%%% 4
\bibitem[\protect\citeauthoryear{\v{C}ert\'{\i}k et~al.}{2020}]{hpVEM}
\begin{barticle}
\bauthor{\bsnm{\v{C}ert\'{\i}k}, \binits{O.}},
\bauthor{\bsnm{Gardini}, \binits{F.}},
\bauthor{\bsnm{Manzini}, \binits{G.}},
\bauthor{\bsnm{Mascotto}, \binits{L.}},
\bauthor{\bsnm{Vacca}, \binits{G.}}:
\batitle{{The {$p$}- and {$hp$}-versions of the virtual element method for
  elliptic eigenvalue problems}}.
\bjtitle{Comput. Math. Appl.}
\bvolume{79}(\bissue{7}),
\bfpage{2035}--\blpage{2056}
(\byear{2020})
\doiurl{10.1016/j.camwa.2019.10.018}
\end{barticle}
\endbibitem

%%% 5
\bibitem[\protect\citeauthoryear{Dassi et~al.}{2020}]{LavadinaEl}
\begin{barticle}
\bauthor{\bsnm{Dassi}, \binits{F.}},
\bauthor{\bsnm{Lovadina}, \binits{C.}},
\bauthor{\bsnm{Visinoni}, \binits{M.}}:
\batitle{{A three-dimensional {H}ellinger-{R}eissner virtual element method for
  linear elasticity problems}}.
\bjtitle{Comput. Methods Appl. Mech. Engrg.}
\bvolume{364},
\bfpage{112910}--\blpage{17}
(\byear{2020})
\doiurl{10.1016/j.cma.2020.112910}
\end{barticle}
\endbibitem

%%% 6
\bibitem[\protect\citeauthoryear{Beir\~{a}o~da Veiga et~al.}{2013}]{BMEl}
\begin{barticle}
\bauthor{\bsnm{Veiga}, \binits{L.}},
\bauthor{\bsnm{Brezzi}, \binits{F.}},
\bauthor{\bsnm{Marini}, \binits{L.D.}}:
\batitle{{Virtual elements for linear elasticity problems}}.
\bjtitle{SIAM J. Numer. Anal.}
\bvolume{51}(\bissue{2}),
\bfpage{794}--\blpage{812}
(\byear{2013})
\doiurl{10.1137/120874746}
\end{barticle}
\endbibitem

%%% 7
\bibitem[\protect\citeauthoryear{Artioli et~al.}{2017}]{LovadinaEl2}
\begin{barticle}
\bauthor{\bsnm{Artioli}, \binits{E.}},
\bauthor{\bsnm{Miranda}, \binits{S.}},
\bauthor{\bsnm{Lovadina}, \binits{C.}},
\bauthor{\bsnm{Patruno}, \binits{L.}}:
\batitle{{A stress/displacement virtual element method for plane elasticity
  problems}}.
\bjtitle{Comput. Methods Appl. Mech. Engrg.}
\bvolume{325},
\bfpage{155}--\blpage{174}
(\byear{2017})
\doiurl{10.1016/j.cma.2017.06.036}
\end{barticle}
\endbibitem

%%% 8
\bibitem[\protect\citeauthoryear{Antonietti et~al.}{2014}]{Inc1}
\begin{barticle}
\bauthor{\bsnm{Antonietti}, \binits{P.F.}},
\bauthor{\bsnm{Veiga}, \binits{L.}},
\bauthor{\bsnm{Mora}, \binits{D.}},
\bauthor{\bsnm{Verani}, \binits{M.}}:
\batitle{{A stream virtual element formulation of the {S}tokes problem on
  polygonal meshes}}.
\bjtitle{SIAM J. Numer. Anal.}
\bvolume{52}(\bissue{1}),
\bfpage{386}--\blpage{404}
(\byear{2014})
\doiurl{10.1137/13091141X}
\end{barticle}
\endbibitem

%%% 9
\bibitem[\protect\citeauthoryear{Beir\~{a}o~da Veiga et~al.}{2018}]{Inc2}
\begin{barticle}
\bauthor{\bsnm{Veiga}, \binits{L.}},
\bauthor{\bsnm{Lovadina}, \binits{C.}},
\bauthor{\bsnm{Vacca}, \binits{G.}}:
\batitle{{Virtual elements for the {N}avier-{S}tokes problem on polygonal
  meshes}}.
\bjtitle{SIAM J. Numer. Anal.}
\bvolume{56}(\bissue{3}),
\bfpage{1210}--\blpage{1242}
(\byear{2018})
\doiurl{10.1137/17M1132811}
\end{barticle}
\endbibitem

%%% 10
\bibitem[\protect\citeauthoryear{Beir\~{a}o~da Veiga et~al.}{2017}]{Inc3}
\begin{barticle}
\bauthor{\bsnm{Veiga}, \binits{L.}},
\bauthor{\bsnm{Lovadina}, \binits{C.}},
\bauthor{\bsnm{Vacca}, \binits{G.}}:
\batitle{{Divergence free virtual elements for the {S}tokes problem on
  polygonal meshes}}.
\bjtitle{ESAIM Math. Model. Numer. Anal.}
\bvolume{51}(\bissue{2}),
\bfpage{509}--\blpage{535}
(\byear{2017})
\doiurl{10.1051/m2an/2016032}
\end{barticle}
\endbibitem

%%% 11
\bibitem[\protect\citeauthoryear{Antonietti et~al.}{2016}]{F1}
\begin{barticle}
\bauthor{\bsnm{Antonietti}, \binits{P.F.}},
\bauthor{\bsnm{Veiga}, \binits{L.}},
\bauthor{\bsnm{Scacchi}, \binits{S.}},
\bauthor{\bsnm{Verani}, \binits{M.}}:
\batitle{{A {$C^1$} virtual element method for the {C}ahn-{H}illiard equation
  with polygonal meshes}}.
\bjtitle{SIAM J. Numer. Anal.}
\bvolume{54}(\bissue{1}),
\bfpage{34}--\blpage{56}
(\byear{2016})
\doiurl{10.1137/15M1008117}
\end{barticle}
\endbibitem

%%% 12
\bibitem[\protect\citeauthoryear{Brezzi and Marini}{2013}]{F2}
\begin{barticle}
\bauthor{\bsnm{Brezzi}, \binits{F.}},
\bauthor{\bsnm{Marini}, \binits{L.D.}}:
\batitle{{Virtual element methods for plate bending problems}}.
\bjtitle{Comput. Methods Appl. Mech. Engrg.}
\bvolume{253},
\bfpage{455}--\blpage{462}
(\byear{2013})
\doiurl{10.1016/j.cma.2012.09.012}
\end{barticle}
\endbibitem

%%% 13
\bibitem[\protect\citeauthoryear{Zhao et~al.}{2016}]{F3}
\begin{barticle}
\bauthor{\bsnm{Zhao}, \binits{J.}},
\bauthor{\bsnm{Chen}, \binits{S.}},
\bauthor{\bsnm{Zhang}, \binits{B.}}:
\batitle{{The nonconforming virtual element method for plate bending
  problems}}.
\bjtitle{Math. Models Methods Appl. Sci.}
\bvolume{26}(\bissue{9}),
\bfpage{1671}--\blpage{1687}
(\byear{2016})
\doiurl{10.1142/S021820251650041X}
\end{barticle}
\endbibitem

%%% 14
\bibitem[\protect\citeauthoryear{Beir\~{a}o~da Veiga et~al.}{2017}]{Wave1}
\begin{barticle}
\bauthor{\bsnm{Veiga}, \binits{L.}},
\bauthor{\bsnm{Mora}, \binits{D.}},
\bauthor{\bsnm{Rivera}, \binits{G.}},
\bauthor{\bsnm{Rodr\'{\i}guez}, \binits{R.}}:
\batitle{{A virtual element method for the acoustic vibration problem}}.
\bjtitle{Numer. Math.}
\bvolume{136}(\bissue{3}),
\bfpage{725}--\blpage{763}
(\byear{2017})
\doiurl{10.1007/s00211-016-0855-5}
\end{barticle}
\endbibitem

%%% 15
\bibitem[\protect\citeauthoryear{Perugia et~al.}{2016}]{Perugia}
\begin{barticle}
\bauthor{\bsnm{Perugia}, \binits{I.}},
\bauthor{\bsnm{Pietra}, \binits{P.}},
\bauthor{\bsnm{Russo}, \binits{A.}}:
\batitle{{A plane wave virtual element method for the {H}elmholtz problem}}.
\bjtitle{ESAIM Math. Model. Numer. Anal.}
\bvolume{50}(\bissue{3}),
\bfpage{783}--\blpage{808}
(\byear{2016})
\doiurl{10.1051/m2an/2015066}
\end{barticle}
\endbibitem

%%% 16
\bibitem[\protect\citeauthoryear{Beir\~{a}o~da Veiga et~al.}{2018a}]{M1}
\begin{barticle}
\bauthor{\bsnm{Veiga}, \binits{L.}},
\bauthor{\bsnm{Brezzi}, \binits{F.}},
\bauthor{\bsnm{Dassi}, \binits{F.}},
\bauthor{\bsnm{Marini}, \binits{L.D.}},
\bauthor{\bsnm{Russo}, \binits{A.}}:
\batitle{{Lowest order virtual element approximation of magnetostatic
  problems}}.
\bjtitle{Comput. Methods Appl. Mech. Engrg.}
\bvolume{332},
\bfpage{343}--\blpage{362}
(\byear{2018})
\doiurl{10.1016/j.cma.2017.12.028}
\end{barticle}
\endbibitem

%%% 17
\bibitem[\protect\citeauthoryear{Beir\~{a}o~da Veiga et~al.}{2018b}]{M2}
\begin{barticle}
\bauthor{\bsnm{Veiga}, \binits{L.}},
\bauthor{\bsnm{Brezzi}, \binits{F.}},
\bauthor{\bsnm{Dassi}, \binits{F.}},
\bauthor{\bsnm{Marini}, \binits{L.D.}},
\bauthor{\bsnm{Russo}, \binits{A.}}:
\batitle{{A family of three-dimensional virtual elements with applications to
  magnetostatics}}.
\bjtitle{SIAM J. Numer. Anal.}
\bvolume{56}(\bissue{5}),
\bfpage{2940}--\blpage{2962}
(\byear{2018})
\doiurl{10.1137/18M1169886}
\end{barticle}
\endbibitem

%%% 18
\bibitem[\protect\citeauthoryear{Berrone et~al.}{2023}]{BerroneE2VEM}
\begin{barticle}
\bauthor{\bsnm{Berrone}, \binits{S.}},
\bauthor{\bsnm{Borio}, \binits{A.}},
\bauthor{\bsnm{Marcon}, \binits{F.}},
\bauthor{\bsnm{Teora}, \binits{G.}}:
\batitle{{A first-order stabilization-free Virtual Element Method}}.
\bjtitle{Applied Mathematics Letters}
\bvolume{142},
\bfpage{108641}
(\byear{2023})
\doiurl{10.1016/j.aml.2023.108641}
\end{barticle}
\endbibitem

%%% 19
\bibitem[\protect\citeauthoryear{Berrone et~al.}{2022}]{BerroneE2VEM2}
\begin{barticle}
\bauthor{\bsnm{Berrone}, \binits{S.}},
\bauthor{\bsnm{Borio}, \binits{A.}},
\bauthor{\bsnm{Marcon}, \binits{F.}}:
\batitle{{Comparison of standard and stabilization free Virtual Elements on
  anisotropic elliptic problems}}.
\bjtitle{Applied Mathematics Letters}
\bvolume{129},
\bfpage{107971}
(\byear{2022})
\doiurl{10.1016/j.aml.2022.107971}
\end{barticle}
\endbibitem

%%% 20
\bibitem[\protect\citeauthoryear{Meng et~al.}{2022}]{EigLowOrder}
\begin{barticle}
\bauthor{\bsnm{Meng}, \binits{J.}},
\bauthor{\bsnm{Wang}, \binits{X.}},
\bauthor{\bsnm{Bu}, \binits{L.}},
\bauthor{\bsnm{Mei}, \binits{L.}}:
\batitle{{A lowest-order free-stabilization Virtual Element Method for the
  Laplacian eigenvalue problem}}.
\bjtitle{Journal of Computational and Applied Mathematics}
\bvolume{410},
\bfpage{114013}
(\byear{2022})
\doiurl{10.1016/j.cam.2021.114013}
\end{barticle}
\endbibitem

%%% 21
\bibitem[\protect\citeauthoryear{Beir\~{a}o~da Veiga
  et~al.}{2023}]{AdaptiveE2VEM}
\begin{barticle}
\bauthor{\bsnm{Veiga}, \binits{L.}},
\bauthor{\bsnm{Canuto}, \binits{C.}},
\bauthor{\bsnm{Nochetto}, \binits{R.H.}},
\bauthor{\bsnm{Vacca}, \binits{G.}},
\bauthor{\bsnm{Verani}, \binits{M.}}:
\batitle{{Adaptive VEM: Stabilization-Free A Posteriori Error Analysis and
  Contraction Property}}.
\bjtitle{SIAM Journal on Numerical Analysis}
\bvolume{61}(\bissue{2}),
\bfpage{457}--\blpage{494}
(\byear{2023})
\doiurl{10.1137/21M1458740}
\end{barticle}
\endbibitem

%%% 22
\bibitem[\protect\citeauthoryear{Lehrenfeld and Stocker}{}]{PaulTreffetz}
\begin{botherref}
\oauthor{\bsnm{Lehrenfeld}, \binits{C.}},
\oauthor{\bsnm{Stocker}, \binits{P.}}:
{Embedded Trefftz discontinuous Galerkin methods}.
International Journal for Numerical Methods in Engineering
\textbf{n/a}(n/a)
\doiurl{10.1002/nme.7258}
\end{botherref}
\endbibitem

%%% 23
\bibitem[\protect\citeauthoryear{Qin}{2005}]{Trefftz}
\begin{barticle}
\bauthor{\bsnm{Qin}, \binits{Q.-H.}}:
\batitle{{Trefftz Finite Element Method and Its Applications}}.
\bjtitle{Applied Mechanics Reviews}
\bvolume{58}(\bissue{5}),
\bfpage{316}--\blpage{337}
(\byear{2005})
\doiurl{10.1115/1.1995716}
{\href{https://arxiv.org/abs/https://asmedigitalcollection.asme.org/appliedmechanicsreviews/article-pdf/58/5/316/5441315/316\_1.pdf}{{https://asmedigitalcollection.asme.org/appliedmechanicsreviews/article-pdf/58/5/316/5441315/316\_1.pdf}}}
\end{barticle}
\endbibitem

%%% 24
\bibitem[\protect\citeauthoryear{Antonietti et~al.}{2020}]{Polylapcian}
\begin{barticle}
\bauthor{\bsnm{Antonietti}, \binits{P.F.}},
\bauthor{\bsnm{Manzini}, \binits{G.}},
\bauthor{\bsnm{Verani}, \binits{M.}}:
\batitle{{The conforming virtual element method for polyharmonic problems}}.
\bjtitle{Comput. Math. Appl.}
\bvolume{79}(\bissue{7}),
\bfpage{2021}--\blpage{2034}
(\byear{2020})
\doiurl{10.1016/j.camwa.2019.09.022}
\end{barticle}
\endbibitem

%%% 25
\bibitem[\protect\citeauthoryear{Antonietti et~al.}{[2022] \copyright
  2022}]{PolyLaplacianPt2}
\begin{bchapter}
\bauthor{\bsnm{Antonietti}, \binits{P.F.}},
\bauthor{\bsnm{Manzini}, \binits{G.}},
\bauthor{\bsnm{Mazzieri}, \binits{I.}},
\bauthor{\bsnm{Scacchi}, \binits{S.}},
\bauthor{\bsnm{Verani}, \binits{M.}}:
\bctitle{{The conforming virtual element method for polyharmonic and
  elastodynamics problems: a review}}.
In: \bbtitle{The Virtual Element Method and Its Applications}.
\bsertitle{SEMA SIMAI Springer Ser.},
vol. \bseriesno{31},
pp. \bfpage{411}--\blpage{451}.
\bpublisher{Springer}, \blocation{???}
(\byear{[2022] \copyright 2022}).
\doiurl{10.1007/978-3-030-95319-5\_10}
\end{bchapter}
\endbibitem

%%% 26
\bibitem[\protect\citeauthoryear{Gopal and Trefethen}{2019}]{LighteningLaplace}
\begin{barticle}
\bauthor{\bsnm{Gopal}, \binits{A.}},
\bauthor{\bsnm{Trefethen}, \binits{L.N.}}:
\batitle{{Solving {L}aplace problems with corner singularities via rational
  functions}}.
\bjtitle{SIAM J. Numer. Anal.}
\bvolume{57}(\bissue{5}),
\bfpage{2074}--\blpage{2094}
(\byear{2019})
\doiurl{10.1137/19M125947X}
\end{barticle}
\endbibitem

%%% 27
\bibitem[\protect\citeauthoryear{Trefethen}{2020}]{LighteningLaplace2}
\begin{barticle}
\bauthor{\bsnm{Trefethen}, \binits{L.N.}}:
\batitle{{Numerical conformal mapping with rational functions}}.
\bjtitle{Comput. Methods Funct. Theory}
\bvolume{20}(\bissue{3-4}),
\bfpage{369}--\blpage{387}
(\byear{2020})
\doiurl{10.1007/s40315-020-00325-w}
\end{barticle}
\endbibitem

%%% 28
\bibitem[\protect\citeauthoryear{Gopal and
  Trefethen}{2019}]{LightningHelmholtz}
\begin{botherref}
\oauthor{\bsnm{Gopal}, \binits{A.}},
\oauthor{\bsnm{Trefethen}, \binits{L.}}:
New laplace and helmholtz solvers.
Proceedings of the National Academy of Sciences
\textbf{116}
(2019)
\end{botherref}
\endbibitem

%%% 29
\bibitem[\protect\citeauthoryear{Brubeck and
  Trefethen}{2022}]{LighteningStokes}
\begin{barticle}
\bauthor{\bsnm{Brubeck}, \binits{P.D.}},
\bauthor{\bsnm{Trefethen}, \binits{L.N.}}:
\batitle{{Lightning {S}tokes solver}}.
\bjtitle{SIAM J. Sci. Comput.}
\bvolume{44}(\bissue{3}),
\bfpage{1205}--\blpage{1226}
(\byear{2022})
\doiurl{10.1137/21M1408579}
\end{barticle}
\endbibitem

%%% 30
\bibitem[\protect\citeauthoryear{Zhu and Nakatsukasa}{2023}]{Yuji}
\begin{botherref}
\oauthor{\bsnm{Zhu}, \binits{W.}},
\oauthor{\bsnm{Nakatsukasa}, \binits{Y.}}:
{Convergence and Near-optimal Sampling for Multivariate Function Approximations
  in Irregular Domains via Vandermonde with Arnoldi}.
arXiv preprint arXiv:2301.12241
(2023)
\end{botherref}
\endbibitem

%%% 31
\bibitem[\protect\citeauthoryear{Gilardi}{2020}]{Gilardi}
\begin{bbook}
\bauthor{\bsnm{Gilardi}, \binits{G.}}:
\bbtitle{{Analisi 3}}.
\bpublisher{McGraw-Hill Education}, \blocation{???}
(\byear{2020})
\end{bbook}
\endbibitem

%%% 32
\bibitem[\protect\citeauthoryear{Ahlfors and Collection}{1979}]{Ahlfors}
\begin{bbook}
\bauthor{\bsnm{Ahlfors}, \binits{L.V.}},
\bauthor{\bsnm{Collection}, \binits{K.M.R.}}:
\bbtitle{{Complex Analysis: An Introduction to The Theory of Analytic Functions
  of One Complex Variable}}.
\bsertitle{International series in pure and applied mathematics}.
\bpublisher{McGraw-Hill Education}, \blocation{???}
(\byear{1979})
\end{bbook}
\endbibitem

%%% 33
\bibitem[\protect\citeauthoryear{Grisvard}{2011}]{Grisvard}
\begin{bbook}
\bauthor{\bsnm{Grisvard}, \binits{P.}}:
\bbtitle{{Elliptic Problems in Nonsmooth Domains}}.
\bsertitle{Classics in Applied Mathematics},
vol. \bseriesno{69},
p. \bfpage{410}.
\bpublisher{Society for Industrial and Applied Mathematics (SIAM),
  Philadelphia, PA}, \blocation{???}
(\byear{2011}).
\doiurl{10.1137/1.9781611972030.ch1} .
\bcomment{Reprint of the 1985 original [ MR0775683], With a foreword by Susanne
  C. Brenner}
\end{bbook}
\endbibitem

%%% 34
\bibitem[\protect\citeauthoryear{Gilbarg and
  Trudinger}{2001}]{GilbargTrudinger}
\begin{bbook}
\bauthor{\bsnm{Gilbarg}, \binits{D.}},
\bauthor{\bsnm{Trudinger}, \binits{N.S.}}:
\bbtitle{{Elliptic Partial Differential Equations of Second Order}}.
\bsertitle{Classics in Mathematics},
p. \bfpage{517}.
\bpublisher{Springer}, \blocation{???}
(\byear{2001}).
\bcomment{Reprint of the 1998 edition}
\end{bbook}
\endbibitem

%%% 35
\bibitem[\protect\citeauthoryear{Evans}{2010}]{Evans}
\begin{bbook}
\bauthor{\bsnm{Evans}, \binits{L.C.}}:
\bbtitle{{Partial Differential Equations}},
\bedition{2}nd edn.
\bsertitle{Graduate Studies in Mathematics},
vol. \bseriesno{19},
p. \bfpage{749}.
\bpublisher{American Mathematical Society, Providence, RI}, \blocation{???}
(\byear{2010}).
\doiurl{10.1090/gsm/019}
\end{bbook}
\endbibitem

%%% 36
\bibitem[\protect\citeauthoryear{Herremans et~al.}{2023}]{Herremans2023}
\begin{botherref}
\oauthor{\bsnm{Herremans}, \binits{A.}},
\oauthor{\bsnm{Huybrechs}, \binits{D.}},
\oauthor{\bsnm{Trefethen}, \binits{L.N.}}:
Resolution of singularities by rational functions
(2023)
\end{botherref}
\endbibitem

%%% 37
\bibitem[\protect\citeauthoryear{Brubeck et~al.}{2021}]{VA}
\begin{barticle}
\bauthor{\bsnm{Brubeck}, \binits{P.D.}},
\bauthor{\bsnm{Nakatsukasa}, \binits{Y.}},
\bauthor{\bsnm{Trefethen}, \binits{L.N.}}:
\batitle{Vandermonde with arnoldi}.
\bjtitle{SIAM Review}
\bvolume{63}(\bissue{2}),
\bfpage{405}--\blpage{415}
(\byear{2021})
\doiurl{10.1137/19M130100X}
{\href{https://arxiv.org/abs/https://doi.org/10.1137/19M130100X}{{https://doi.org/10.1137/19M130100X}}}
\end{barticle}
\endbibitem

%%% 38
\bibitem[\protect\citeauthoryear{Brenner and Scott}{2007}]{BrennerFEM}
\begin{bbook}
\bauthor{\bsnm{Brenner}, \binits{S.}},
\bauthor{\bsnm{Scott}, \binits{R.}}:
\bbtitle{{The Mathematical Theory of Finite Element Methods}}
vol. \bseriesno{15}.
\bpublisher{Springer}, \blocation{???}
(\byear{2007})
\end{bbook}
\endbibitem

%%% 39
\bibitem[\protect\citeauthoryear{Trezzi and Zerbinati}{2023}]{GitHub}
\begin{botherref}
\oauthor{\bsnm{Trezzi}, \binits{M.L.}},
\oauthor{\bsnm{Zerbinati}, \binits{U.}}:
{LightningVEM}.
GitHub
(2023).
\doiurl{10.5281/zenodo.8192892} .
\url{https://github.com/UZerbinati/LightningVEM}
\end{botherref}
\endbibitem

%%% 40
\bibitem[\protect\citeauthoryear{Xue et~al.}{2023}]{AAA}
\begin{botherref}
\oauthor{\bsnm{Xue}, \binits{Y.}},
\oauthor{\bsnm{Waters}, \binits{S.L.}},
\oauthor{\bsnm{Trefethen}, \binits{L.N.}}:
{Computation of 2D Stokes flows via lightning and AAA rational approximation}
(2023)
\end{botherref}
\endbibitem

%%% 41
\bibitem[\protect\citeauthoryear{Boffi et~al.}{2022}]{EigReview}
\begin{bbook}
\bauthor{\bsnm{Boffi}, \binits{D.}},
\bauthor{\bsnm{Gardini}, \binits{F.}},
\bauthor{\bsnm{Gastaldi}, \binits{L.}}:
In: \beditor{\bsnm{Antonietti}, \binits{P.F.}},
\beditor{\bsnm{Veiga}, \binits{L.}},
\beditor{\bsnm{Manzini}, \binits{G.}} (eds.)
\bbtitle{Virtual Element Approximation of Eigenvalue Problems},
pp. \bfpage{275}--\blpage{320}.
\bpublisher{Springer},
\blocation{Cham}
(\byear{2022}).
\doiurl{10.1007/978-3-030-95319-5_7}
\end{bbook}
\endbibitem

%%% 42
\bibitem[\protect\citeauthoryear{Boffi et~al.}{2020}]{EigParameterDep}
\begin{barticle}
\bauthor{\bsnm{Boffi}, \binits{D.}},
\bauthor{\bsnm{Gardini}, \binits{F.}},
\bauthor{\bsnm{Gastaldi}, \binits{L.}}:
\batitle{{Approximation of PDE eigenvalue problems involving parameter
  dependent matrices}}.
\bjtitle{Calcolo}
\bvolume{57}(\bissue{4}),
\bfpage{41}
(\byear{2020})
\doiurl{10.1007/s10092-020-00390-6}
\end{barticle}
\endbibitem

%%% 43
\bibitem[\protect\citeauthoryear{Gardini and Vacca}{2017}]{Gardini}
\begin{barticle}
\bauthor{\bsnm{Gardini}, \binits{F.}},
\bauthor{\bsnm{Vacca}, \binits{G.}}:
\batitle{{Virtual element method for second-order elliptic eigenvalue
  problems}}.
\bjtitle{IMA Journal of Numerical Analysis}
\bvolume{38}(\bissue{4}),
\bfpage{2026}--\blpage{2054}
(\byear{2017})
\doiurl{10.1093/imanum/drx063}
\end{barticle}
\endbibitem

\end{thebibliography}

\end{document}